\documentclass[12pt]{amsart}
\numberwithin{equation}{section}
\usepackage{amsmath,amsthm,amsfonts,amscd,eucal}

\usepackage{xypic}
\usepackage{graphicx}

\usepackage{amssymb}
\def\cb{{\mathcal B}}

\def\ce{{\mathcal E}}

\def\ch{{\mathcal H}}

\def\ck{{\mathcal K}}

\def\cs{{\mathcal S}}
\def\ct{{\mathcal T}}

\def\ga{{\mathfrak A}} 
\def\gpb{{\mathfrak b}}

\def\gf{{\mathfrak F}}

\def\gz{{\mathfrak Z}}

\def\bc{{\mathbb C}}

\def\bm{{\mathbb M}}
\def\bn{{\mathbb N}}
\def\bp{{\mathbb P}}

\def\bz{{\mathbb Z}}

\def\a{\alpha}
\def\b{\beta}
\def\g{\gamma}  \def\G{\Gamma}

\def\eps{\varepsilon}

\def\l{\lambda}

\def\s{\sigma} 

\def\f{\varphi}  
  
\def\om{\omega} \def\Om{\Omega}

\newtheorem{thm}{Theorem}[section]
\newtheorem{lem}[thm]{Lemma}

\newtheorem{prop}[thm]{Proposition}

\theoremstyle{definition}

\newtheorem{defin}[thm]{Definition}

\def\aut{\mathop{\rm Aut}}

\def\tr{\mathop{\rm Tr}}

\def\idd{{1}\!\!{\rm I}}

\begin{document}

\title[boolean stochastic processes]
{a note on boolean stochastic processes}
\author{Francesco Fidaleo}
\address{Francesco Fidaleo\\
Dipartimento di Matematica \\
Universit\`{a} di Roma Tor Vergata\\
Via della Ricerca Scientifica 1, Roma 00133, Italy} \email{{\tt
fidaleo@mat.uniroma2.it}}
\date{\today}

\begin{abstract}
For the quantum stochastic processes generated by the Boolean 
Commutation Relations, we prove the following version of De Finetti Theorem: each of such Boolean process is exchangeable if and only if it is independent and identically distributed with respect to the tail algebra.\\
\vskip0.1cm\noindent
{\bf Mathematics Subject Classification}: 60G09, 46L53, 46L05, 46L30, 46N50.\\
{\bf Key words}:  Exchangeability; Non commutative probability and statistics;
$C^{*}$--algebras; States; Applications to Quantum Physics.
\end{abstract}

\maketitle

\section{introduction}
\label{sec1}

In classical probability theory De Finetti--Hewitt--Savange Theorem asserts that a stochastic process is exchangeable or symmetric if and only if it is a convex combination of independent and identically distributed stochastic processes, or equivalently it is independent and identically distributed with respect to the tail algebra, see e.g. \cite{DeF, HS, Ka}. Such a result is partially generalised to non commutative, or quantum, case for stochastic processes arising from infinite tensor product or from the Fermi (cf. Canonical Anti--commutation Relations) algebra, see e.g. \cite{ABCL, CrF, CrF1, St2}. Combining the results in \cite{CrF1, St2} with the local structure of the Canonical Commutation Relations algebra (cf. \cite{LRT}), it can be
also straightforwardly seen that De Finetti--Hewitt--Savange still holds for the case of quantum fields describing Bose particles.

The situation arising from the general quantum setting is rather complicated, and  in \cite{CrF1} the first systematic attempt to investigate the structure of non commutative exchangeable stochastic processes is done. The first, and maybe the main difficulty one meets in quantum case, is that a conditional expectation onto a subalgebra preserving a given state is not automatically guaranteed as in the classical case. The reader is referred to \cite{Sr} for details, and a wide literature (even if still partial) cited therein, about this point. Indeed, the quantum stochastic processes arising from Boolean Commutation Relations (cf. \cite{BGS}) provide a class of examples suitable for the investigations. The aim of the present note is in fact the following. On one hand, it is possible to see that
there are many examples of Boolean processes for which there exists no conditional expectation onto the tail algebra preserving the associated state, even if the latter is Abelian (but it is never a sub algebra of the centre of the Gelfand--Naimatk--Segal realisation of the process). On the other hand, it is possible to show by direct calculation, that a Boolean stochastic process is exchangeable if and only if it is independent and identically distributed with respect the tail algebra. This result, combined with Proposition 7.3 of \cite{CrF1} which asserts that any Boolean exchangeable process is convex combination of extremal (i.e. ergodic) ones, provides another remarkable example of quantum stochastic processes for which the complete form of De Finetti Theorem holds true. 

\section{preliminaries}
\label{sec2}

Consider the group $\bp_J:=\bigcup_{\{I\subset J\mid I\,\text{finite}\,\}}\bp_I$, made of all the permutations of the set $J$ moving only a finite numbers of indices, together with an action
$$
\g\in \bp_J\mapsto \a_g\in\aut(\ga)
$$
of $\bp_J$ by $*$--automorphisms of $\ga$. A state $\f\in\cs(\ga)$ is said to be {\it symmetric} if it is invariant w.r.t. the action $\a$ of
$\bp_J$. The set of the symmetric states is denoted by $\cs_{\bp_J}(\ga)$. When $\ga$ is unital, $\cs_{\bp_J}(\ga)$ is convex and compact in the weak--$*$ topology whose convex boundary is denoted by $\ce_{\bp_J}(\ga)$.

Following \cite{AFL, CrF1}, a {\it stochastic process} labelled by the index set $J$ is a quadruple
$\big(\ga,\ch,\{\iota_j\}_{j\in J},\Om\big)$, where $\ga$ is a $C^{*}$--algebra, $\ch$ is an Hilbert space,
the $\iota_j$'s are $*$--homomorphisms of $\ga$ in $\cb(\ch)$, and
$\Om\in\ch$ is a unit vector, cyclic for  the von Neumann algebra
$M:=\bigvee_{j\in J}\iota_j(\ga)$ naturally acting on $\ch$.
The process is said to be {\it exchangeable} if, for each $g\in\bp_J$, $n\in\bn$, $j_1,\ldots j_n\in J$, $A_1,\ldots A_n\in\ga$
$$
\langle\iota_{j_1}(A_1)\cdots\iota_{j_n}(A_n)\Om,\Om\rangle
=\langle\iota_{g(j_{1})}(A_1)\cdots\iota_{g(j_{n})}(A_n)\Om,\Om\rangle.
$$
The stochastic process as above is uniquely determined up to a natural unitary equivalence relation. To simplify, we limit the matter to the unital case, that is when $\ga$ has the unity $\idd$ and $\iota_j(\idd)=I$. 

Let a stochastic process $\big(\ga,\ch,\{\iota_j\}_{j\in J},\Om\big)$ be given, together with its corresponding state $\f$ on the free product $C^*$--algebra $\gf$ in unital or not unital case,
of $\ga$. Define the {\it tail algebra} of the process under consideration as
\begin{equation*}
\gz^\perp_\f:=\bigcap_{\begin{subarray}{l}I\subset J,\,
I \text{finite} \end{subarray}}\bigg(\bigcup_{\begin{subarray}{l}K\bigcap I=\emptyset,
\\\,\,\,K \text{finite} \end{subarray}}\bigg(\bigvee_{k\in K}\iota_k(\ga)\bigg)\bigg)''\,.
\end{equation*}
We provide the definition of conditionally independent and identically distributed process w.r.t. the tail algebra
$\gz^\perp_\f$ which is useful in quantum case.
\begin{defin}
\label{cocaind}
The stochastic process described by the state $\f\in\cs(\gf)$, is {\it conditionally independent and identically distributed}
w.r.t. the tail algebra if there exists a conditional expectation 
$$
E_\f:\bigvee_{j\in J}\iota_j(\ga)\to \gz^\perp_\f 
$$ 
preserving the vector state $\langle\,{\bf\cdot}\,\Om_\f,\Om_\f\rangle$ such that,
\begin{itemize}
\item[(i)] $\langle XY\Om_\f,\Om_\f\rangle=\langle E_\f(X)E_\f(Y)\Om_\f,\Om_\f\rangle$,
for each finite subsets $I,K\subset J$, $I\cap K=\emptyset$, and
$$
X\in\bigg(\bigvee_{i\in I}\iota_i(\ga)\bigg)\bigvee\gz^\perp_\f\,,
Y\in\bigg(\bigvee_{k\in K}\iota_k(\ga)\bigg)\bigvee\gz^\perp_\f\,;
$$
\item[(ii)]$E_\f(\iota_i(A))=E_\f(\iota_k(A))$
for each $i,k\in J$ and $A\in\ga$.
\end{itemize}
\end{defin}
The following Lemma, useful for application (see e.g. the proof of Theorem 5.4 in \cite{CrF1}), provides an {\it a--priori} stronger condition for being independent w.r.t. the tail algebra.
\begin{lem} 
\label{ellek}
The condition (i) of Definition \ref{cocaind} is equivalent to
\begin{itemize}
\item[(i1)] $\langle X_1X_2\cdots X_n\Om_\f,\Om_\f\rangle=\langle E_\f(X_1)E_\f(X_2)\cdots E_\f(X_n)\Om_\f,\Om_\f\rangle$,
for each finite subsets $I_k,I_l\subset J$: $I_k\cap I_l=\emptyset$, $k,l=1,2,\dots,n$, $k\neq l$, with
$$
X_k\in\bigg(\bigvee_{i\in I_k}\iota_i(\ga)\bigg)\bigvee\gz^\perp_\f\,,\quad k=1,2,\dots,n\,.
$$
\end{itemize}
\end{lem}
\begin{proof}
The proof follows by the repeated application of (i), and the bimodule property of the conditional expectation. We get
\begin{align*}
&\langle X_1X_2X_3\cdots X_{n-2}X_{n-1}X_n\Om_\f,\Om_\f\rangle\\
=&\langle E_\f(X_1)E_\f(X_{2}X_3\cdots X_{n-2}X_{n-1}X_n)\Om_\f,\Om_\f\rangle\\
=&\langle E_\f(E_\f(X_1)X_{2}X_3\cdots X_{n-2}X_{n-1}X_n)\Om_\f,\Om_\f\rangle\\
=&\langle E_\f(E_\f(X_1)X_{2})E_\f(X_3\cdots X_{n-2}X_{n-1}X_n)\Om_\f,\Om_\f\rangle\\
=&\langle E_\f(X_1)E_\f(X_{2})E_\f(X_3\cdots X_{n-2}X_{n-1}X_n)\Om_\f,\Om_\f\rangle\\
&\cdots\qquad\qquad\qquad\cdots\qquad\qquad\qquad\cdots\\
=&\langle E_\f(X_1)\cdots E_\f(X_{n-2})E_\f(X_{n-1}X_n)\Om_\f,\Om_\f\rangle\\
=&\langle E_\f((E_\f(X_1)\cdots E_\f(X_{n-2})X_{n-1})X_n)\Om_\f,\Om_\f\rangle\\
=&\langle E_\f(E_\f(X_1)\cdots E_\f(X_{n-2})X_{n-1})E_\f(X_n)\Om_\f,\Om_\f\rangle\\
=&\langle E_\f(X_1)\cdots E_\f(X_{n-2})E_\f(X_{n-1})E_\f(X_n)\Om_\f,\Om_\f\rangle\,.
\end{align*}
\end{proof}

\section{boolean stochastic processes}
\label{sec5}
Let $\ch$ be a complex Hilbert space. Recall that the Boolean Fock space over $\ch$ (cf. \cite{BGS}) is given by $\G(\ch):=\mathbb{C}\oplus \ch$,
where the vacuum vector $\Om$ is $(1,0)$.
On $\Gamma(\ch)$ we define the creation and annihilation operators, respectively given for $f\in \ch$, by
$$
b^\dagger(f)(\alpha\oplus g):=0\oplus \alpha f,\,\,\,\, b(f)(\alpha\oplus g):=\langle g,f\rangle_\ch \oplus 0,\,\,\, \alpha\in\mathbb{C},\, g\in\ch.
$$
They are mutually adjoint, and satisfy the following relations for $f,g\in \ch$,
$$
b(f)b^\dagger(g)=\langle g, f\rangle_\ch \langle\,{\bf\cdot}\,,\Om\rangle\Om\,,\quad
b^\dagger(f)b(g)=\langle\,{\bf\cdot}\,,0\oplus g\rangle 0\oplus f\,.
$$
To simplify the matter, we consider the unital stochastic Boolean processes. We also reduce the investigation to the separable case, that is $J=\bn$.
As shown in Section 7 of \cite{CrF1}, the unital $C^*$--algebras acting on $\Gamma(\ell^2(\bn))$ 
generated by the annihilators $\{b(f)\mid f\in\ell^2(\bn)\}$, or equally well by the selfadjoint part of annihilators 
$\{b(f)+b^\dagger(f)\mid f\in\ell^2(\bn)\}$,
coincides with 
$$
\gpb=\ck(\ell^2(\{\#\}\cup\bn))+\bc I\,,
$$
$\ck(\ch)$ being the $C^*$--algebra of all the compact operators acting on the Hilbert space $\ch$. Under this isomorphism, annihilators and creators are expressed by the system of matrix--units as follows:
$$
\varepsilon_{\#j}=b_j\,,\,\, \varepsilon_{j\#}=b^\dagger_j\,,\,\,
\varepsilon_{\#\#}=b_ib^\dagger_i\,,\quad \varepsilon_{ij}=b^\dagger_ib_j\,,\quad i,j\in\bn\,.
$$
It is possible to show that the universal $C^*$--algebra generated by the Boolean Commutation Relations still coincides (i.e. is isomorphic) with $\gpb$, see Section 7 of \cite{CrF1}. Thus, we directly refer to $\gpb$ as the {\it Boolean algebra}.
Let 
$$
A=\begin{pmatrix} 
	 a &b \\
	 c & d\\
     \end{pmatrix}
     \oplus \b\in\bm_2(\bc)\bigoplus\bc=:\ga\,.
$$
For each $j\in\bn$, the embeddings $\{\iota_j\}_{j\in\bn}$ generating any Boolean process are then given by
$$
\iota_j(A):=a\eps_{\#\#}+b\eps_{\#j}+c\eps_{j\#}+d\eps_{jj}+\b P_{\bn\backslash\{j\}}\,,
$$
$\{\eps_{mn}\mid m,n\in \{\#\}\cup\bn\}$ being the canonical system of matrix--units for $\cb(\ell^2(\{\#\}\cup\bn))$, and 
$P_{\bn\backslash\{j\}}$ the orthogonal projection onto the subspace $\ell^2(\bn\backslash\{j\})$. The Boolean processes will be identified as states on the Boolean algebra $\gpb$. Amomg them, the exchangeable ones are precisely those which are invariant by the natural action of the permutation group $\bp_\bn$ on $\gpb$.

We report the result in \cite{CrF1} (cf. Proposition 7.3) which asserts that the compact convex set of the symmetric states on $\gpb$ is a Choquet Simplex made of a segment.
\begin{prop}
\label{oilf}
We have for the compact convex set of the symmetric states,
$$
\cs_{\bp_{\bz}}(\gpb)=\{\g\om_\#+(1-\g)\om_\infty\mid\g\in[0,1]\}\,,
$$
where $\om_\#=\langle\,{\bf\cdot}\,e_\#,e_\#\rangle$ is the Fock vacuum state, and
\begin{equation*}
\om_\infty(A+aI):=a\,,\quad A\in \ck(\ell^2(\{\#\}\cup\bz))\,, a\in\bc\,.
\end{equation*}
\end{prop}
Consider any general state $\om\in\cs(\gpb)$, it can be uniquely written as 
\begin{equation}
\label{oif1}
\om=\g\psi_T+(1-\g)\om_\infty\,,\quad T\in\ct(\ell^2(\{\#\}\cup\bz))_{+,1}\,, \g\in[0,1]\}\,.
\end{equation}
Here, $\ct(\ell^2(\{\#\}\cup\bz))_{+,1}$ denotes the set of positive normalized trace--class operators acting on $\ell^2(\{\#\}\cup\bz)$, with $\psi_T(A)=\tr(TA)$ and ''$\tr$'' stands for the unnormalized trace.

Now we show that
there are plenty of Boolean processes for which the tail algebra is not expected, that is no conditional expectation onto such an algebra preserves the state corresponding to the process under consideration. Yet, a Boolean process is exchangeable if and only if it is independent and identically distributed w.r.t. the tail algebra. To see this, we show by explicit computation that the states which  are preserved by some conditional expectation cannot be identically distributed, except for the symmetric ones.
To this end, we fix $T\in\ct(\ell^2(\{\#\}\cup\bz))_{+,1}$ with
$$
T=\sum_{k\in{\bf n}}\l_k\langle \,{\bf\cdot}\,,\xi_k\rangle\xi_k\,,
$$ 
where ${\bf n}$ is a countable set of cardinality $\tr(s(T))$, and $s(T)$ is the support--projection of $T$ in $\cb(\ell^2(\{\#\}\cup\bz))$.
\begin{prop}
\label{tabal}
For each state $\om$ given in \eqref{oif1}, the following assertions hold true.

If $\g=0$ then $\pi_{\om}(\gpb)''=\bc=\gz^\perp_{\om}$. 

If $\g=1$ then
$$
\pi_{\om}(\gpb)''=\cb(\ell^2(\{\#\}\cup\bz))\bigotimes I_{\ell^2({\bf n})}\,,
$$ 
acting on $\cb(\ell^2(\{\#\}\cup\bz)\bigotimes \ell^2({\bf n}))$, and
$$
\gz^\perp_{\om}=\left(\bc P_\#\oplus\bc P_\#^\perp\right)\bigotimes I_{\ell^2({\bf n})}\,.
$$
If $0<\g<1$ then
$$
\pi_{\om}(\gpb)''=\left(\cb(\ell^2(\{\#\}\cup\bz))\bigotimes I_{\ell^2({\bf n})}\right)\bigoplus\bc\,,
$$
acting on $\cb(\ell^2(\{\#\}\cup\bz)\bigotimes \ell^2({\bf n}))\bigoplus\bc$, and
$$
\gz^\perp_{\om}=\left[\left(\bc P_\#\oplus\bc P_\#^\perp\right)\bigotimes I_{\ell^2({\bf n})}\right]\bigoplus\bc\,.
$$
\end{prop}
\begin{proof}
The form of the GNS representation $\pi_\om$ is immediately deduced from those of $\psi_T$ and $\om_\infty$. Concerning the tail algebra, the matter is reduced to the nontrivial case $\g=1$ for which $\om=\psi_T$. We get with $J_n:=\{1,\dots,n\}$, 
$P_\#:=\langle\,{\bf\cdot}\, e_\#,e_\#\rangle e_\#$, and considering $\cb(\ell^2(K))$ a non unital subalgebra of 
$\cb(\ell^2( \{\#\}\cup\bz))$ in a canonical way for any $K\subset \{\#\}\cup\bz$,
\begin{align*}
\gz^\perp_{\psi_T}=&\bigg(\bigcap_{n\in\bn}\cb(\ell^2(\{\#\}\cup J_n))\bigoplus\bc P_{\ell^2(\bn\backslash J_n)}\bigg)
\bigotimes I_{\ell^2({\bf n})}\\
=&\left(\bc P_\#\oplus\bc P_\#^\perp\right)\bigotimes I_{\ell^2({\bf n})}\,,
\end{align*}
\end{proof}
Concerning the conditional expectations onto $\gz^\perp_{\om}$ preserving the state $\om$, the case with $\g=0$ is trivial. The case $0<\g<1$ can be easily reduced to the case $\g=1$ as, if $\om=\g\psi_T+(1-g)\om_\infty$ is a nontrivial convex combination, for 
$X=A\oplus a\in\gz^\perp_{\om}$, $E(X)=F(A)\oplus a$ is a conditional expectation onto the tail algebra provided $F:\cb(\ell^2(\{\#\}\cup \bn))\rightarrow\gz^\perp_{\psi_T}$ is any conditional expectation onto $\gz^\perp_{\psi_T}$. Thus, the unique nontrivial case is to consider directly the case $\om=\psi_T$ onto $\ck(\ell^2(\{\#\}\cup\bn))$. By neglecting the unessential multiplicity, it can easily seen that
each conditional expectation $F$ onto $\gz^\perp_{\om_\xi}$ satisfies $F=F\circ E$, with
$$
E(A)=\om_\#(A)P_\#+P_\#^\perp AP_\#^\perp\,,\quad A\in \cb(\ell^2(\{\#\}\cup \bn))\,.
$$
Then we conclude that any conditional expectation $F=F_\f$ as above 
assumes the form
\begin{equation}
\label{mpast}
F_\f(A)=\om_\#(A)P_\#+\f(P_\#^\perp AP_\#^\perp)P_\#^\perp\,,\quad A\in \cb(\ell^2(\{\#\}\cup \bn))\,,
\end{equation} 
where $\f$ is any state, not necessarily normal, on $\cb(\ell^2(\bn))$, the last viewed again as a non unital subalgebra of 
$\cb(\ell^2(\{\#\}\cup\bn))$.
\begin{prop}
\label{1z1}
Fix a state $\psi_T\in\cs(\ck(\ell^2(\{\#\}\cup \bn)))$. There exists 
a conditional expectation $F_\f$ onto $\gz^\perp_{\psi_T}$ given in \eqref{mpast} 
preserving the state $\psi_T\in\cs(\ck(\ell^2(\{\#\}\cup \bn)))$ if and only if $e_\#$ is an eigenvector of $T$.
\end{prop}
\begin{proof}
Suppose that $e_\#$ is an eigenvalue of $T$ with necessarily $\om_\#(T)$ as eigenvalue. If  $\om_\#(T)=1$, then $T=P_\#$ or equivalently $\psi_T=\om_\#$. Thus,  
any conditional expectation in \eqref{mpast} preserves $\psi_T$. If $e_\#$ is an eigenvalue for $T$ but $\om_\#(T)<1$, then 
$T=\om_\#(T)P_\#+\tilde T$,
where $\psi_{\tilde T}$ is a positive functional
with support $0<s(\tilde T)\leq P_\#^\perp$, and 
$\psi_{\tilde T}=\psi_{\tilde T}(P_\#^\perp\,{\bf\cdot}\,P_\#^\perp)$.
We compute with
$$
\f=\frac1{1-\om_\#(T)}\psi_{\tilde T}\lceil_{\cb(\ell^2(\bn))}\,,
$$
and for each $X\in\cb(\ell^2(\{\#\}\cup \bn))$, 
\begin{align*}
\psi_T(F_\f(X))=&\om_\#(X)\psi_T(P_\#)+\f(P_\#^\perp XP_\#^\perp)\psi_T(P_\#^\perp)\\
=&\om_\#(X)\om_\#(T)+\f(P_\#^\perp XP_\#^\perp)(1-\om_\#(T))\\
=&\om_\#(X)\om_\#(T)+\psi_{\tilde T}(X)=\psi_T(X)\,,
\end{align*}
that is, $\psi_T$ is expected. Suppose now that $e_\#$ is not an eigenvalue of $T$, and consider the subset $J\subset {\bf n}$ such that 
$\langle e_\#,\xi_j\rangle\neq0$, which is nonvoid. Then
$$
\inf_{j\in J}\l_j<\max_{j\in J}\l_j=\l_{j_0}
$$
for some $j_0\in J$, otherwise $e_\#$ would be an eigenvalue of $T$. Put 
$X:=\langle\,{\bf\cdot}\,,\xi_{j_0}\rangle e_\#$ where $\xi_{j_0}$ is a unit eigenvector with eigenvalue $\l_{j_0}$, and compute for each conditional expectation $F_\f$,
$$
\frac{\psi_T(F_\f(X))}{\psi_T(X)}=\sum_k\frac{\l_k}{\l_{j_0}}|\langle e_\#,\xi_k\rangle|^2
<\sum_k|\langle e_\#,\xi_k\rangle|^2=\om_\#(s(T))\leq1\,.
$$
Thus, $\psi_T$ cannot be expected. 
\end{proof}
Proposition \ref{1z1} provides examples of quantum stochastic processes for which, contrarily to the classical case,
the condition to be independent and identically distributed w.r.t. the tail algebra (cf. Definition \ref{cocaind}), cannot be formulated in the general case, without mentioning the {\it a--priori} existence of a preserving conditional expectation. 

We are now ready to prove the main result of the present note.
\begin{thm}
A Boolean process is exchangeable if and only if it is independent and identically distributed w.r.t the tail algebra.
\end{thm}
\begin{proof}
For a state $\om=\g\psi_T+(1-\g)\om_\infty$, $\g\in[0,1]$, the case $\g=0$ is trivial and the case $0<\g<1$ easily follows from that when $\g=1$. So we reduce the matter to $\g=1$, the symmetric situation corresponding to $\om=\om_\#$. It is then enough to put
$\om=\psi_T$ and see first that in the symmetric case that is when $\psi_T=\om_\#$, it is independent and identically distributed w.r.t. the tail algebra. Then in non symmetric case, either $\psi_T$ is not expected hence the corresponding process cannot be identically distributed w.r.t. the tail algebra, or when it is expected we show that the corresponding process cannot be identically distributed.

Concerning the symmetric case $\om=\om_\#$, 
choose any state $\s\in\cs\big(\cb(\ell^2(\bn))/\ck(\ell^2(\bn))$, $\cb(\ell^2(\bn))/\ck(\ell^2(\bn))$ being the Calkin algebra with 
$\pi:\cb(\ell^2(\bn))\rightarrow\cb(\ell^2(\bn))/\ck(\ell^2(\bn))$ the canonical projection.
The singular state $\f:=\s\circ\pi\in\cs(\cb(\ell^2(\bn)))$ is invariant under the natural action of $\bp_\bn$ on $\ell^2(\bn)$. For each $\f$ as above, the conditional expectation $F_\f$ preserves $\om_\#$, and is invariant w.r.t. the action $\a$ of the permutations by construction. Thus, $\om_\#$ is identically distributed. Now for $I,K\subset\bn$ with $I\cap K=\emptyset$, consider $A\in\cb(\ell^2(\{\#\}\cup I))$, $B\in\cb(\ell^2(\{\#\}\cup K))$, and $a,b\in\bc$. 
We have for the the generic element $X\in\big(\bigvee_{i\in I}\iota_i(\ga)\big)\bigvee\gz^\perp_\f$, 
$Y\in\big(\bigvee_{i\in K}\iota_i(\ga)\big)\bigvee\gz^\perp_\f$,
\begin{align*}
X=&\om_\#(A)P_\#+P_\#AP_I+P_IAP_\#+P_IAP_I+aP_{\bn\backslash I}\,,\\
Y=&\om_\#(B)P_\#+P_\#BP_K+P_KBP_\#+P_KBP_K+bP_{\bn\backslash K}\\
F_\f(X)=&\om_\#(A)P_\#+\f(A+aP_{\bn\backslash I})P_\bn\,,\\
F_\f(Y)=&\om_\#(B)P_\#+\f(B+bP_{\bn\backslash I})P_\bn\\
XY=&\om_\#(A)\om_\#(B)P_\#+\om_\#(A)P_\#BP_K+bP_\#AP_I\\
+&\om_\#(B)P_IAP_\#+P_IAP_\#BP_K+bP_IAP_I\\
+&aP_KBP_\#+aP_KBP_K+abP_{\bn\backslash(I\cup K)}\,.
\end{align*}
We then get $\langle XYe_\#,e_\#\rangle=\langle F_\f(X)F_\f(Y)e_\#,e_\#\rangle$, that is $\om_\#$ is independent and identically distributed w.r.t. the tail algebra.

Suppose now that $\psi_T\in\cs(\ck(\ell^2(\{\#\}\cup\bz))$ is a generic state such that 
$\psi_T$ is not symmetric but expected. By Proposition \ref{1z1}, it happens if and only if $e_\#$ is an eigenvector of $T$ with eigenvalue
necessarily $\om_\#(T)<1$.
In this situation, the conditional expectation $F$ onto the tail algebra preserving $\psi_T$ is given by
$$
F(X)=\om_\#(X)P_\#+\frac{\psi_T(X)-\om_\#(T)\om_\#(X)}{1-\om_\#(T)}P_\#^\perp\,.
$$
By Lebesgue Dominated Convergence Theorem, we get
$$
\lim_i\psi_T(b^\dagger_ib_i)=\lim_i\sum_k\l_k|\xi_k(i)|^2=\sum_k\l_k\lim_i|\xi_k(i)|^2=0\,.
$$
Thus, if $\psi_T$
would be identically distributed, this would imply $\psi_T(b^\dagger_ib_i)=0$, $i\in\bn$. 
As $F (b^\dagger_ib_i)=\frac{\psi_T(b^\dagger_ib_i)}{1-\om_\#(T)}P_\#^\perp$, again we have $F(b^\dagger_ib_i)=0$, $i\in\bn$.
Being $F$ a normal map in this case, we get
$$
I=F(I)=F\bigg(P_\#+\sum_{i\in\bn}b^\dagger_ib_i\bigg)=P_\#+\sum_{i\in\bn}F(b^\dagger_ib_i)=P_\#\,,
$$
which is a contradiction.
\end{proof}

\end{document}